\documentclass{amsart}

\usepackage{amsmath,amsfonts,amsopn,amssymb}
\usepackage{graphicx,subfig}
\usepackage{bm}

\makeatletter
\@namedef{subjclassname@2020}{\emph{2020} Mathematics Subject Classification}
\makeatother

\theoremstyle{theorem}
\newtheorem{proposition}{Proposition}
\newtheorem{lemma}[proposition]{Lemma}

\newcommand\bv{\bm{v}}
\newcommand\bw{\bm{w}}
\newcommand\cH{\mathcal{H}}

\usepackage[normalem]{ulem}
\usepackage{color}

\def\vdual{v_{\Delta}}
\def\wdual{w_{\Delta}}
\def\D{\Omega}
\def\bd{\partial}
\def\aj{\tilde{a}_{\D_j}}

\begin{document}

\title[Correction]
{Correction to: A dual iterative substructuring method with a small penalty parameter}

\author{Chang-Ock Lee}
\address[Chang-Ock Lee]{Department of Mathematical Sciences\\
KAIST\\
Daejeon 34141, Korea}
\email{colee@kaist.edu}

\author{Eun-Hee Park}
\address[Eun-Hee Park]{Division of Liberal Studies\\
Kangwon National University\\
Samcheok 25913, Korea}
\email{eh.park@kangwon.ac.kr}

\author{Jongho Park}
\address[Jongho Park]{Department of Mathematical Sciences\\
KAIST\\
Daejeon 34141, Korea}
\email{jongho.park@kaist.ac.kr}



\maketitle

% Abstract
\begin{abstract}
In this corrigendum, we offer a correction to~[J.\ Korean.\ Math.\ Soc., 54~(2017), pp.~461--477].
We construct a counterexample for the strengthened Cauchy--Schwarz inequality used in the original paper.
In addition, we provide a new proof for Lemma~5 of the original paper, an estimate for the extremal eigenvalues of the standard unpreconditioned FETI-DP dual operator. 
\end{abstract}

In the first and second authors' previous work~\cite{LP:2017}, the strengthened Cauchy--Schwarz inequality used for~\cite[Eq.~(3.8)]{LP:2017} is incorrect and consequently, the statement of~\cite[Lemma~4]{LP:2017} needs to be corrected.
We present a new proof for~\cite[Lemma~5]{LP:2017}, that does not use~\cite[Lemma~4]{LP:2017}.
All notations are adopted from the original paper~\cite{LP:2017}.

In the paragraph containing~\cite[Eq.~(3.8)]{LP:2017}, it was claimed that by deriving a strengthened Cauchy-Schwarz inequality in a similar way to Lemma~4.3 in~\cite{LP:2012},
  it is shown that there exists a constant $\gamma$ such that
\begin{equation*}
\quad  2 \tilde{a}(v_I + \vdual, v_c) \geq - \gamma ( \tilde{a}(v_I + \vdual, v_I + \vdual) + \tilde{a}(v_c, v_c) ),
\end{equation*}
where $0<\gamma<1$ is independent of $H$ and $h$.
That is, the above inequality is true when there exists a constant $\gamma$ such that
\begin{equation}\label{Ineq:Strengthened-CS}
|\tilde{a}(v_I + \vdual, v_c)| \leq \gamma \left(\tilde{a}(v_I + \vdual, v_I + \vdual)\right)^{1/2} \left( \tilde{a}(v_c, v_c)\right)^{1/2}, \\
\end{equation}
where $0<\gamma<1$ is independent of $h$ and $H$.

On the other hand, a specific function $w = w_{I} + w_{c} + \wdual$ can be constructed, for which $\gamma$ approaches 1 as $H$ decreases. In fact, it suffices to characterize such $\wdual$ because $w_{I}$ and $w_{c}$ in~\eqref{Ineq:Strengthened-CS} are determined by $\wdual$ in terms of the discrete $\tilde{a}$-harmonic extension $\mathcal{H}^{c}(w_{\Delta})$.

% Proposition: No gamma
\begin{proposition}
There is no $\gamma~(0< \gamma < 1)$, independent of $h$ and $H$, satisfying~\eqref{Ineq:Strengthened-CS}.
\end{proposition}
\begin{proof}
Noting that $\mathcal{H}^{c}(\vdual)$ in $X^{c}_{h}$ is $\tilde{a}(\cdot,\cdot)$-orthogonal to all the functions which vanish at the interface nodes except for the subdomain corners, we have that
\begin{align*}\label{Eq:fromA-orthgonality}
\tilde{a} (v_I + \vdual, v_c)
&= \tilde{a}\left(\mathcal{H}^{c}(\vdual) - v_c, v_c \right) \\
&= \tilde{a}\left(\mathcal{H}^{c}(\vdual), v_c\right) - \tilde{a}(v_c, v_c)\\
&= -\tilde{a}(v_c, v_c),
\end{align*}
which implies that for $\tilde{a}(v_I + \vdual, v_I + \vdual) \neq 0$, the estimate \eqref{Ineq:Strengthened-CS} is equivalent to
\begin{equation}\label{Ineq:equivform}
   \frac{\tilde{a}(v_c, v_c)}{\tilde{a}(v_I + \vdual, v_I + \vdual)}\leq \gamma^2,
\end{equation}
where $\gamma<1$ is independent of $h$ and $H$.

Next, let us divide $\D=(0,1)^2$ into $1/H \times 1/H$ square subdomains with a side length $H$. Each subdomain is partitioned into $2 \times H/h \times H/h$ uniform right triangles. Associated with such a triangulation, we select the function $w$ in $X^{c}_{h}$ such that $w$  is a conforming $\mathbb{P}_{1}$ element function in each subdomain, and 
$\wdual = 1$ at all the nodes on the interface except for the subdomain corners.
Hence, for $w_{c}$ and $w_{I}$ that are computed by the discrete harmonic extension of $\wdual$, it is observed that
\begin{subequations}
\begin{align}
\label{Prop:wc}
  & w_{c}=1 \text{ at all the subdomain corners }x_{k} \text{ that are not on } \partial\Omega,\\
  & w_{I}=1 \text{ in } \Omega_j \text{ for }\bd\D_j \cap \bd\D = \varnothing,\label{Prop:wi}
\end{align}
\end{subequations}
which imply that
\begin{equation}\label{Prop-w:floating-subD}
  w \equiv 1 \text{ in all subdomains whose boundary does not touch } \bd\D.
\end{equation}

Let us first estimate $\tilde{a}(w_{c}, w_{c})$ in~\eqref{Ineq:equivform}. Using~\eqref{Prop:wc}, we have that
\begin{equation*}
  \tilde{a}(w_{c}, w_{c}) = \sum_{k=1}^{\left(1/H-1\right)^2} \tilde{a}(\phi_{c,k},\phi_{c,k}) = 4\left(\frac{1}{H}-1\right)^{2},
\end{equation*}
where $\phi_{c,k}$ is the nodal basis function associated with the corner $x_k$.
We next look over $\tilde{a}(w_{I} + \wdual, w_{I} + \wdual)$ based on the fact that, for $\bd\D_j \cap \bd\D = \varnothing$
\begin{equation}\label{Est:floating-subD}
  \aj(w_{I} + \wdual,w_{I} + \wdual) = \int_{\D_j}|\nabla (w_{I} + \wdual)|^2 dx = \int_{\D_j}|\nabla w_{c}|^2 dx = 4,
\end{equation}
which follows from~\eqref{Prop-w:floating-subD}. Hence it suffices to estimate $\aj(w_{I} + \wdual, w_{I} + \wdual)$ for the following two cases:
\begin{itemize}
  \item[(i)] only one of the edges of the subdomain $\D_j$ is on $\bd \D$.
  \item[(ii)] two edges of the subdomain $\D_j$ are on $\bd \D$.
\end{itemize}
Here, the number of subdomains corresponding to the cases~(i) and~(ii) is $4\left(\frac{1}{H}-2\right)$ and $4$, respectively.
Let us take $H/h=3$ to focus only on the dependence of $\gamma$ on either $H$ or $h$. By finding the discrete local harmonic extensions for the cases~(i) and~(ii), it is computed directly that
\begin{equation}
\label{Prop-w}
\aj(w_{I} + \wdual, w_{I} + \wdual) =  \begin{cases} \frac{17}{4}  \quad \text{for the case (i)}, \\ \frac{14}{4}  \quad \text{for the case (ii)}. \end{cases}
\end{equation}
Then by using~\eqref{Est:floating-subD} and~\eqref{Prop-w}, it follows that
\begin{equation}\label{Est:vivd}
\begin{split}
\tilde{a}(w_{I} + \wdual, w_{I} + \wdual)
  &= \left( \sum_{\substack{j\text{ for}\\ \bd\D_{j}\cap\bd\D=\varnothing}}
  + \sum_{\substack{j\text{ for}\\ \bd\D_{j}\cap\bd\D\neq\varnothing}} \right) \aj(w_{I} + \wdual, w_{I} + \wdual) \\
  &=4\left(\frac{1}{H}-2\right)^2 + 17\left(\frac{1}{H}-2\right) + 14.
\end{split}
\end{equation}
Finally, from~\eqref{Prop:wc} and~\eqref{Est:vivd}, it is confirmed that for a function $w$ given above,
\begin{equation*}
 \lim_{ H\rightarrow 0 } \frac{\tilde{a}(w_{c}, w_{c})}{\tilde{a}(w_{I} + \wdual, w_{I} + \wdual)} = 1,
\end{equation*}
which implies that \eqref{Ineq:equivform} does not hold.
\end{proof}

In~\cite[Lemma~5]{LP:2017}, the extremal eigenvalues of the FETI-DP dual operator $F = B_{\Delta}S^{-1}B_{\Delta}^T$ were estimated using~\cite[Lemma~4]{LP:2017}, estimates for the extremal eigenvalues of $S$.
Since~\cite[Lemma~4]{LP:2017} is incorrect, we provide a new estimate for $F$ that does not utilize~\cite[Lemma~4]{LP:2017}.
We assume that each subdomain $\Omega_j$ is the union of elements in a conforming coarse mesh $\mathcal{T}_H$ of $\Omega$.
First, we consider the following Poincar\'{e}-type inequality that generalizes~\cite[Proposition~3]{LP:2017}.

% Lemma: Poincare inequality
\begin{lemma}
\label{Lem:Poincare}
For any $v_j \in X_h^j$, let $I_j^H v_j$ be the linear coarse interpolation of $v_j$ such that $I_j^H v_j = v_j$ at vertices of a subdomain $\Omega_j \subset \mathbb{R}^d$.
Then we have
\begin{equation*}
|v_j|_{H^1 (\Omega_j)}^2 \gtrsim \begin{cases}
\, H^{-1} \left( 1 + \ln \frac{H}{h} \right)^{-1} \| v_j - I_j^H v_j \|_{L^2 (\partial \Omega_j)}^2 &\textrm{ for } d= 2, \\
\, h^{-1} \left( \frac{H}{h} \right)^{-2} \| v_j - I_j^H v_j \|_{L^2 (\partial \Omega_j)}^2 &\textrm{ for } d= 3.
\end{cases}
\end{equation*} 
\end{lemma}
\begin{proof}
Since the both sides of the inequality do not change if a constant is added to $v_j$, we may assume that $v_j$ has the zero average, so that the following Poincar\'{e} inequality holds:
\begin{equation}
\label{Poincare1}
\| v_j \|_{H^1 (\Omega_j)} \lesssim | v_j |_{H^1 (\Omega_j)},
\end{equation}
where $\| \cdot \|_{H^1 (\Omega_j)}$ is the weighted $H^1$-norm on $\Omega_j$ given by
\begin{equation*}
\| v_j \|_{H^1 (\Omega_j)}^2 = |v_j|_{H^1 (\Omega_j)}^2 + \frac{1}{H^2} \| v_j \|_{L^2 (\Omega_j)}^2.
\end{equation*}
Since $I_j^H v_j$ attains its extremum at vertices, we have
\begin{equation} \begin{split}
\label{Poincare2}
\| v_j - I_j^H v_j \|_{L^2 (\partial \Omega_j)}
&\lesssim H^{\frac{d-1}{2}} \| v_j - I_j^H v_j \|_{L^{\infty}(\partial \Omega_j )} \\
&\leq H^{\frac{d-1}{2}} \left( \| v_j \|_{L^{\infty}(\partial \Omega_j)} + \| I_j^H v_j \|_{L^{\infty}(\partial \Omega_j)} \right) \\
&\lesssim H^{\frac{d-1}{2}} \| v_j \|_{L^{\infty}(\partial \Omega_j)}.
\end{split} \end{equation}
Let $\cH_j v_j$ be the generalized harmonic extension of $v_j|_{\partial \Omega_j}$ introduced in~\cite{XZ:1998} such that $\cH_j v_j = v_j$ on $\partial \Omega_j$ and
\begin{equation}
\label{general_harmonic}
\| \cH_j v_j \|_{H^{1}(\Omega_j)} = \min_{\substack{w_j \in H^1 (\Omega_j) \\w_j = v_j \textrm{ on } \partial \Omega_j}} \| w_j \|_{H^1 (\Omega_j)}.
\end{equation}
Then it follows that
\begin{subequations}
\label{Poincare3}
\begin{align}
H^{d-1} \|v_j \|_{L^{\infty}(\partial \Omega_j)}^2
&\leq H^{d-1} \| \cH_j v_j \|_{L^{\infty} (\Omega_j)}^2 \nonumber \\
&\lesssim C_d(H,h) \| \cH_j v_j \|_{H^1 (\Omega_j)}^2 \label{Poincare3_1}\\
&\leq C_d(H,h) \| v_j \|_{H^1 (\Omega_j)}^2 \label{Poincare3_2}\\
&\lesssim C_d (H,h) |v_j|_{H^1 (\Omega_j)}^2, \label{Poincare3_3}
\end{align}
\end{subequations}
where
\begin{equation*}
C_d (H,h) = \begin{cases}
H\left(1 + \ln \frac{H}{h} \right) &\textrm{ for } d = 2, \\
h \left( \frac{H}{h} \right)^{2} &\textrm{ for } d = 3,
\end{cases}
\end{equation*}
and~\eqref{Poincare3_1} is due to the discrete Sobolev inequality~\cite[Lemma~2.3]{BX:1991}.
Also~\eqref{general_harmonic} and~\eqref{Poincare1} are used in~\eqref{Poincare3_2} and~\eqref{Poincare3_3}, respectively.
Combination of~\eqref{Poincare2} and~\eqref{Poincare3} completes the proof.
\end{proof}

Note that Lemma~\ref{Lem:Poincare} reduces to~\cite[Proposition~3]{LP:2017} when $v_j$ vanishes at vertices of $\Omega_j$ so that $I_j^H v_j = 0$.
Using Lemma~\ref{Lem:Poincare}, we obtain the following estimate for $F$.

% Correction of Lemma 5
\begin{proposition}
\label{Prop:F}
For $F = B_{\Delta}S^{-1}B_{\Delta}^T$, we have
\begin{equation*}
\underline{C}_F \lambda^T \lambda \lesssim \lambda^T F \lambda
\lesssim \overline{C}_F \lambda^T \lambda \quad \forall \lambda
\end{equation*}
where
\begin{equation*}
\underline{C}_F = h^{2-d} \textrm{ for }d=2,3,
\end{equation*}
and
\begin{equation*}
\overline{C}_F = \begin{cases}
\, \left( \frac{H}{h}\right) \left( 1 + \ln \frac{H}{h} \right) & \textrm{ for } d = 2, \\
\,  h^{-1} \left( \frac{H}{h}\right)^2 & \textrm{ for } d = 3.
\end{cases}
\end{equation*}
Consequently, the condition number of $F$ satisfies the following bound:
\begin{equation*}
\kappa (F) \lesssim \begin{cases}
\, \left( \frac{H}{h}\right) \left( 1 + \ln \frac{H}{h} \right) & \textrm{ for } d = 2, \\
\, \left( \frac{H}{h}\right)^2 & \textrm{ for } d = 3.
\end{cases}
\end{equation*}
\end{proposition}
\begin{proof}
As the derivation of the maximum eigenvalue of $S$ in the original paper~\cite{LP:2017} is correct, the derivation of $\underline{C}_F$ is also correct.
Thus, we only estimate $\overline{C}_F$ in the following.
We note that our proof closely follows~\cite[Theorem~4.5]{MT:2001}.

Similarly to~\cite[Theorem~4.4]{MT:2001}, it suffices to prove that
\begin{equation}
\label{F1}
(B_{\Delta} \bv_{\Delta})^T (B_{\Delta} \bv_{\Delta}) \lesssim \overline{C}_F \bv_{\Delta}^T S \bv_{\Delta} \quad \forall \bv_{\Delta}.
\end{equation}
If~\eqref{F1} were true, we get the desired result as follows:
\begin{equation*} \begin{split}
\lambda^T F \lambda &= \max_{\bv_{\Delta} \neq 0} \frac{\left( (B_{\Delta} \bv_{\Delta})^T \lambda \right)^2}{\bv_{\Delta}^T S \bv_{\Delta}} \\
&\lesssim \overline{C}_F \max_{B_{\Delta}\bv_{\Delta} \neq 0 } \frac{\left( (B_{\Delta}\bv_{\Delta})^T \lambda \right)^2}{(B_{\Delta} \bv_{\Delta})^T B_{\Delta} \bv_{\Delta}} \\
&\leq \overline{C}_F \max_{\mu \neq 0} \frac{(\mu^T \lambda)^2}{\mu^T \mu} \\
&= \overline{C}_F \lambda^T \lambda,
\end{split} \end{equation*}
where we used~\cite[Lemma~4.3]{MT:2001} in the first equality.

Take any $v_{\Delta}$ and its discrete $\tilde{a}$-harmonic extension $v = \mathcal{H}^c(v_{\Delta})$.
Let $w = v - I^H v$, where $I^H v$ is the linear coarse interpolation of $v$ onto $\mathcal{T}^H$ such that $I^H v = v$ at the subdomain vertices.
We write $w = w_I + w_{\Delta}$.
Since $I^H v$ is continuous along $\Gamma$, we have $B_{\Delta} \bw_{\Delta} = B_{\Delta} \bv_{\Delta}$.
Then it follows that
\begin{equation*} \begin{split}
(B_{\Delta} \bv_{\Delta})^T (B_{\Delta} \bv_{\Delta})
&= (B_{\Delta} \bw_{\Delta})^T (B_{\Delta} \bw_{\Delta}) \\
&= \sum_{j<k} \left( \bw_{\Delta}^{(j)}\big|_{\Gamma_{jk}} - \bw_{\Delta}^{(k)}\big|_{\Gamma_{jk}} \right)^T \left( \bw_{\Delta}^{(j)}\big|_{\Gamma_{jk}} - \bw_{\Delta}^{(k)}\big|_{\Gamma_{jk}} \right) \\
&\lesssim \sum_{j < k} \left( \left( \bw_{\Delta}^{(j)}\big|_{\Gamma_{jk}} \right)^T \bw_{\Delta}^{(j)}\big|_{\Gamma_{jk}} + \left( \bw_{\Delta}^{(k)}\big|_{\Gamma_{jk}} \right)^T \bw_{\Delta}^{(k)}\big|_{\Gamma_{jk}}\right) \\
&\lesssim \sum_{j=1}^{N_s} (\bw_{\Delta}^{(j)})^T (\bw_{\Delta}^{(j)}) \\
&\lesssim h^{1-d} \sum_{j=1}^{N_s} \| w \|_{L^2 (\partial \Omega_j)}^2 \\
&\lesssim \overline{C}_F \bv_{\Delta}^T S \bv_{\Delta},
\end{split} \end{equation*}
where the last inequality is due to Lemma~\ref{Lem:Poincare}.
\end{proof}

It must be mentioned that the conclusion of Proposition~\ref{Prop:F} agrees with Lemma~5 of the original paper~\cite{LP:2017}.
Since the conclusion of~\cite[Lemma~5]{LP:2017} is true, it requires no additional correction in the remaining part of that paper.

For the sake of completeness, we present a correct estimate for the extremal eigenvalues of $S$ that replaces~\cite[Lemma~4]{LP:2017}.

% Proposition: Correction of Lemma 4
\begin{proposition}
\label{Prop:S}
For $S = A_{\Delta \Delta} - A_{I \Delta}^T A_{I I}^{-1} A_{I \Delta}$, we have
\begin{equation*}
\underline{C}_S \bv_{\Delta}^T \bv_{\Delta} \lesssim \bv_{\Delta}^T S \bv_{\Delta}
\lesssim \overline{C}_S \bv_{\Delta}^T \bv_{\Delta} \quad \forall \bv_{\Delta},
\end{equation*}
where
\begin{equation*}
\underline{C}_S = \begin{cases}
\,  H h \left( 1 + \ln \frac{H}{h} \right)^{-1} & \textrm{ for } d = 2, \\
\,  h^3 & \textrm{ for } d = 3,
\end{cases}
\end{equation*}
and
\begin{equation*}
\overline{C}_S = h^{d-2} \textrm{ for }d=2,3.
\end{equation*}
\end{proposition}
\begin{proof}
Since the derivation of $\overline{C}_S$ in the original paper~\cite{LP:2017} is correct, we only consider an estimate for $\underline{C}_S$.
Take any $\bv_{\Delta}$ and its corresponding finite element function $v_{\Delta}$.
Let $v = \cH^c (v_{\Delta})$ be the discrete $\tilde{a}$-harmonic extension of $v_{\Delta}$.
Proceeding as in~\cite[Lemma~4.11]{TW:2005}, we get
\begin{equation*} \begin{split}
\bv_{\Delta}^T \bv_{\Delta} &\lesssim h^{1-d} \sum_{j=1}^{N_s} \| v_{\Delta} \|_{L^2 (\partial \Omega_j)}^2 \\
&\lesssim Hh^{1-d} \sum_{j=1}^{N_s} \left( | v |_{H^1 (\Omega_j)}^2 + H^{-2} \|v \|_{L^2 (\Omega_j)}^2 \right) \\
&= Hh^{1-d} \bv_{\Delta}^T S \bv_{\Delta} + H^{-1}h^{1-d} \| v \|_{L^2 (\Omega)}^2.
\end{split} \end{equation*}
Note that we cannot apply the discrete Poincar\'{e} inequality~\cite[Lemma~5.1]{BL:2005} in each subdomain $\Omega_j$ since $\cH v_{\Delta}$ does not vanish at the subdomain vertices in general.

It remains to show that
\begin{equation}
\label{S1}
\| v \|_{L^2 (\Omega)}^2 \lesssim \begin{cases}
\,  \left( 1 + \ln \frac{H}{h} \right) \bv_{\Delta}^T S \bv_{\Delta} & \textrm{ for } d = 2, \\
\, \frac{H}{h} \bv_{\Delta}^T S \bv_{\Delta} & \textrm{ for } d = 3.
\end{cases}
\end{equation}
Let $I^H v$ be the linear nodal interpolation of $v$ onto the coarse mesh $\mathcal{T}_H$.
Since $I^H v$ is continuous along the subdomain interfaces $\Gamma$,
we can apply the Poincar\'{e} inequality to obtain
\begin{equation*}
\| I^H v \|_{L^2 (\Omega)} \lesssim | I^H v|_{H^1 (\Omega)}.
\end{equation*}
Then it follows that
\begin{equation*} \begin{split}
\| v \|_{L^2 (\Omega)}^2 
&\lesssim \| v - I^H v \|_{L^2 (\Omega)}^2 + \| I^H v \|_{L^2 (\Omega)}^2 \\
&\lesssim \| v - I^H v \|_{L^2 (\Omega)}^2 + | I^H v |_{H^1 (\Omega)}^2 \\
&\lesssim \begin{cases}
\,  \left( 1 + \ln \frac{H}{h} \right) \bv_{\Delta}^T S \bv_{\Delta} & \textrm{ for } d = 2, \\
\,  \frac{H}{h} \bv_{\Delta}^T S \bv_{\Delta} & \textrm{ for } d = 3,
\end{cases}
\end{split} \end{equation*}
where the last inequality is due to~\cite[Remark~4.13]{TW:2005} for $d=2$ and~\cite[Lemma~4.12]{TW:2005} for $d=3$, respectively.
\end{proof}

%% Acknowledgement
%\subsection*{Acknowledgements}
%Chang-Ock Lee's work was supported by the National Research Foundation of Korea~(NRF) grant funded by the Korea government~(MSIT)~(No. NRF-2017R1A2B4011627).
%Eun-Hee Park's work was supported by TO BE MODIFIED.
%Jongho Park's work was supported by NRF funded by the Ministry of Education~(No. 2019R1A6A3A01092549).

% References
\bibliographystyle{siam}
\bibliography{refs_LP2017}
\end{document}